\RequirePackage{fix-cm}
\documentclass[twocolumn]{svjour3}          
\smartqed  
\usepackage{graphicx}
\usepackage{mathrsfs}
\usepackage{cite}
\usepackage{amsmath,amssymb,amsfonts}
\usepackage{algorithmic}
\usepackage{textcomp}
\usepackage{mfl}

\usepackage{euscript}
\usepackage{latexsym}
\usepackage{xspace}

\usepackage[T1]{fontenc}

\usepackage{url,color}
\usepackage{mathrsfs}
\def\BibTeX{{\rm B\kern-.05em{\sc i\kern-.025em b}\kern-.08em
    T\kern-.1667em\lower.7ex\hbox{E}\kern-.125emX}}

\usepackage{tikz}
\usetikzlibrary{arrows,chains,matrix,positioning,scopes}
\def\BibTeX{{\rm B\kern-.05em{\sc i\kern-.025em b}\kern-.08em
    T\kern-.1667em\lower.7ex\hbox{E}\kern-.125emX}}

\makeatletter
\tikzset{join/.code=\tikzset{after node path={%
\ifx\tikzchainprevious\pgfutil@empty\else(\tikzchainprevious)%
edge[every join]#1(\tikzchaincurrent)\fi}}}
\makeatother

\tikzset{>=stealth',every on chain/.append style={join},
         every join/.style={->}}
\tikzstyle{labeled}=[execute at begin node=$\scriptstyle,
   execute at end node=$]

\usepackage{ctable}

\newtheorem{Def}{Definition}
\newtheorem{Thm}{Theorem}

\newtheorem{Cor}[Thm]{Corollary}
\newtheorem{Lem}[Thm]{Lemma}
\newtheorem{Pro}[Thm]{Proposition}

\begin{document}

\title{Syntactic characterizations of classes of first-order\\ structures in mathematical fuzzy logic
}
\subtitle{}


\author{Guillermo Badia         \and
        Vicent Costa \and Pilar Dellunde \and Carles Noguera 
}

\authorrunning{G. Badia, V. Costa, P. Dellunde, C. Noguera} 

\institute{G. Badia \at
Department of Knowledge-Based Mathematical Systems \\
Johannes Kepler University Linz, Austria,\\
School of Historical and Philosophical Inquiry\\
 University of Queensland, Australia,\\
\email{guillebadia89@gmail.com}           
 \and
V. Costa \at
Department of Philosophy \\    
Universitat Aut\`onoma de Barcelona, Catalonia,\\
\email{Vicente.Costa@uab.cat}  
\and
P. Dellunde \at 
Department of Philosophy \\ 
Universitat Aut\`onoma de Barcelona, Catalonia,\\
Artificial Intelligence Research Institute (IIIA -- CSIC), Catalonia,\\
Barcelona Graduate School of Mathematics, Catalonia\\    
\email{pilar.dellunde@uab.cat}
\and
C. Noguera \at
Institute of Information Theory and Automation,\\    
Czech Academy of Sciences, Czech Republic,\\
\email{noguera@utia.cas.cz}           
}

\date{This paper is dedicated to Llu\'is Godo in the occasion of his 60th birthday}

\maketitle

\begin{abstract}
This paper is a contribution to graded model theory, in the context of mathematical fuzzy logic. We study characterizations of classes of graded structures in terms of the syntactic form of their first-order axiomatization. We focus on classes given by universal and universal-existential sentences. In particular, we prove two amalgamation results using the technique of diagrams in the setting of structures valued on a finite MTL-algebra, from which analogues of the  \L o\'s--Tarski and the Chang--\L o\'s--Suszko preservation theorems follow.
\end{abstract}

\section{Introduction}

Graded model theory is the generalized study, in mathematical fuzzy logic (MFL), of the construction and classification of graded structures. The field was properly started in~\cite{Ha06} and has received renewed attention in recent years \cite{Badia-Noguera:Fraisse,Ba13,Cintula-Metcalfe:HerbrandLPAR,Cintula-Diaconescu-Metcalfe:SkolemLPAR,CoDe16,De11,De14}. Part of the programme of graded model theory is to find non-classical analogues of results from classical model theory (e.g.,\cite{h, sacks, Chang73}). This will not only provide generalizations of classical theorems but will also provide insight into what avenues of research are particular to classical first-order logic and do not make sense in a broader setting.

On the other hand, classical model theory was developed together with the analysis of some very relevant mathematical structures. In consequence, its principal results provided a logical interpretation of such structures. Thus, if we want model theory's idiosyncratic interaction with other disciplines to be preserved, the redefinition of the fundamental notions of graded model theory cannot be obtained from directly fuzzifying every classical concept. Quite the contrary, the experience acquired in the study of different structures, the results obtained using specific classes of structures, and the potential overlaps with other areas should determine the light the main concepts of graded model theory have to be defined in. It is in this way that several fundamental concepts of the model theory of mathematical fuzzy logic have already appeared in the literature.

The goal of this paper is to give syntactic characterizations of classes of graded structures; more precisely, we want to study which kind of formulas can be used to axiomatize certain classes of structures based on finite (expansions of) \MTL-chains. Traditional examples of such sort of results are preservation theorems in classical model theory, which, in general, can be obtained as consequences of certain amalgamation properties (cf. \cite{h}). We provide some amalgamation results using the technique of diagrams which will allow us to establish analogues of the \L o\'s--Tarski preservation theorem~\cite[Theorem 6.5.4]{h} and the Chang--\L o\'s--Suszko theorem~\cite[Theorem 6.5.9]{h}. 

This is not the first work that addresses a model-theoretic study of the preservation and characterization of classes of fuzzy structures. Indeed, Bagheri and Moniri \cite{Ba13} have obtained results for the particular case of continuous model theory by working over the standard \MV-algebra $[0,1]_{\text{\L}}$ and with a predicate language enriched with a truth-constant for each element of $[0,1]_{\text{\L}}$. In that context, they characterize universal theories in terms of the preservation under substructures \cite[Prop. 5.1]{Ba13}, and prove versions of the Tarski--Vaught theorem \cite[Prop. 4.6]{Ba13} and of the Chang--\L o\'s--Suszko theorem \cite[Prop. 5.5]{Ba13}.

The connection between classical model theory and the study of classes of fuzzy structures needs to be clarified. Namely, as explained and developed in previous papers~\cite{Cintula-EGGMN:DistinguishedSemantics,DeGaNo16,DeGaNo18}, there is a translation of fuzzy structures into classical many-sorted structures, more precisely, two-sorted structures with one sort for the first-order domain and another accounting for truth-values in the algebra. Such connection certainly allows to directly import to the fuzzy setting several classical results, but, as already noted in the mentioned papers, it does not go a long way. Indeed, the translation does not preserve the syntactical complexity of sentences (regarding quantifiers) and, hence, it cannot be used for syntactically-sensitive results, such as those studied in the present paper. 

The paper is structured as follows: in section 1, we introduce the syntax and semantics of fuzzy predicate logics. In section~\ref{s:prelim},
several fuzzy model-theoretic notions such as homomorphisms or the method of diagrams are presented. In section~\ref{s:universal}, we study the preservation of universal formulas, obtain an existential form of amalgamation and derive from it an analogue of the \L o\'s--Tarski theorem. In section~\ref{s:universal-existential}, we study classes given by universal-existential sentences by showing that such formulas are preserved under unions of chain, obtaining another corresponding amalgamation result and a version of Chang--\L o\'s--Suszko preservation theorem. We end with some concluding remarks and suggestions for lines of further research.

\section{Preliminaries}\label{section preliminaries}\label{s:prelim}

In this section we introduce the syntax and semantics of fuzzy predicate logics, and recall the basic results on diagrams we will use in the paper. We use the notation and definitions of the Handbook of Mathematical Fuzzy Logic~\cite{CiFeHaNo11}.

\begin{Def} \emph{(Syntax of Predicate Languages)}
 A \emph{predicate language} $\mathcal{P}$ is a  triple $\left\langle Pred_{\mathcal{P}},Func_{\mathcal{P}},Ar_{\mathcal{P}} \right\rangle$, where $Pred_{\mathcal{P}}$ is a non-empty set of \emph{predicate symbols}, $Func_{\mathcal{P}}$ is a set of \emph{function symbols} (disjoint from $Pred_{\mathcal{P}}$), and $Ar_{\mathcal{P}}$ represents the \emph{arity function}, which assigns a natural number to each predicate symbol or function symbol. We call this natural number the \emph{arity} of the symbol. The predicate symbols with arity zero are called \emph{truth constants}, while the function symbols whose arity is zero are named \emph{object constants} (\emph{constants} for short). 
\end{Def} 

$\mathcal{P}$-terms, $\mathcal{P}$-formulas, $\forall_n$ and $\exists_n$ $\mathcal{P}$-formulas, and the notions of free occurrence of a variable, open formula, substitutability, and sentence are defined as in classical predicate logic. A theory is a set of sentences. When it is clear from the context, we will refer to $\mathcal{P}$-terms and $\mathcal{P}$-formulas simply as \emph{terms} and \emph{formulas}. 

Let MTL stand for the monoidal t-norm based logic introduced by Esteva and Godo~\cite{go}.
Throughout the paper, we consider the predicate logic MTL$\forall$ (for a definition of the axiomatic system for MTL$\forall$ we refer the reader to~\cite[Def. 5.1.2, Ch. I]{CiFeHaNo11}).
Let us recall that the deduction rules of MTL$\forall$ are those of MTL and the rule of generalization: from $\varphi$ infer $(\forall x)\varphi$. 
The definitions of proof and provability are analogous to the classical ones. 
We denote by $\Phi\vdash_{\text{MTL}\forall}\varphi$ the fact that $\varphi$ is provable in  MTL$\forall$ from the set of formulas $\Phi$. 
For the sake of clarity, when it is clear from the context we will write $\vdash$ to refer to $\vdash_{\text{MTL}\forall}$. The algebraic semantics of MTL$\forall$ is based on {\em \MTL-algebras}~\cite{go}.

$\alg{A}$ is called an {\em \MTL-chain} if its underlying lattice is linearly ordered. Since it is costumary to consider fuzzy logicsin languages expanding that of \MTL, henceforth, we will confine our attention to algebras which are expansions of \MTL-chains of such kind and just call them {\em chains}.

\begin{Def} \emph{(Semantics of Predicate Fuzzy Logics~\cite[Def. 5.2.1, Ch. I]{CiFeHaNo11})}  \label{evaluation} Consider a predicate language $\mathcal{P}=\langle Pred_{\mathcal{P}}, Func_{\mathcal{P}}, Ar_{\mathcal{P}} \rangle$ and let \textbf{A} be a chain. We define an $\textbf{A}$\emph{-structure} $\mathrm{\mathbf{M}}$ for $\mathcal{P}$ as a pair $\model{M} = \tuple{\alg{A},\struct{M}}$ where $$\struct{M}= \langle M, (P_M)_{P\in Pred}, (F_M)_{F\in Func} \rangle,$$ where $M$ is a non-empty domain, $P_{\mathrm{\mathbf{M}}}$ is an $n$-ary fuzzy relation for each $n$-ary predicate symbol, i.e., a function from $M^n$ to $A$, identified with an element of $\textbf{A}$ if $n=0$; and $F_{\mathrm{\mathbf{M}}}$ is a function from $M^n$ to $M$, identified with an element of $M$ if $n=0$. As usual, if\/ $\mathrm{\mathbf{M}}$ is an $\textbf{A}$-structure for $\mathcal{P}$, an $\mathrm{\mathbf{M}}$-evaluation of the object variables is a mapping $v$ assigning to each object variable an element of $M$. The set of all object variables is denoted by $Var$. If $v$ is an $\mathrm{\mathbf{M}}$-evaluation, $x$ is an object variable and $d\in M$, we denote by $v[x\mapsto d]$ the $\mathrm{\mathbf{M}}$-evaluation so that $v[x\mapsto d](x)=d$ and $v[x\mapsto d](y)=v(y)$ for $y$ an object variable such that $y\not=x$. If\/ $\mathrm{\mathbf{M}}$ is an $\textbf{A}$-structure and $v$ is an $\mathrm{\mathbf{M}}$-evaluation, we define the \emph{values} of terms and the \emph{truth values} of formulas in $M$ for an evaluation $v$ recursively as follows:
 
\begin{itemize}
\item[] $\|x\|^{{\boldsymbol{A}}}_{\mathbf{M},v}=v(x)$;

\item[] $\|F(t_1,\ldots,t_n)\|^{{\boldsymbol{A}}}_{\mathbf{M},v}=F_{\mathbf{M}}(\|t_1\|^{{\boldsymbol{A}}}_{\mathbf{M},v},\ldots,\|t_n\|^{{\boldsymbol{A}}}_{\mathbf{M},v})$, \newline for $F\in Func$;

\item[] $\|P(t_1,\ldots,t_n)\|^{{\boldsymbol{A}}}_{\mathbf{M},v}=P_{\mathbf{M}}(\|t_1\|^{{\boldsymbol{A}}}_{\mathbf{M},v},\ldots,\|t_n\|^{{\boldsymbol{A}}}_{\mathbf{M},v})$, \newline for $P\in Pred$;

\item[] $\|c(\varphi_1,\ldots,\varphi_n)\|^{{\boldsymbol{A}}}_{\mathbf{M},v}=\circ^{\boldsymbol{A}}(\|\varphi_1\|^{{\boldsymbol{A}}}_{\mathbf{M},v},\ldots,\|\varphi_n\|^{{\boldsymbol{A}}}_{\mathbf{M},v})$, \newline for $\circ\in\mathcal{L}$;

\item[] $\|(\forall x)\varphi\|^{{\boldsymbol{A}}}_{\mathbf{M},v}=inf_{\leq^{\boldsymbol{A}}}\{\|\varphi\|^{{\boldsymbol{A}}}_{\mathbf{M},v[x\rightarrow d]}\mid d\in M\}$;

\item[] $\|(\exists x)\varphi\|^{{\boldsymbol{A}}}_{\mathbf{M},v}=sup_{\leq^{\boldsymbol{A}}}\{\|\varphi\|^{{\boldsymbol{A}}}_{\mathbf{M},v[x\rightarrow d]}\mid d\in M\}$.
\end{itemize}
\end{Def} 

For a set of formulas $\Phi$, we write $\|\Phi\|^{\boldsymbol{A}}_{\mathbf{M},v}=1$,  if $\|\varphi\|^{\boldsymbol{A}}_{\mathbf{M},v}=1$ for every $\varphi\in\Phi$. We denote by $\|\varphi\|^{\boldsymbol{A}}_{\mathbf{M}}=1$ the fact that $\|\varphi\|^{\boldsymbol{A}}_{\mathbf{M},v}=1$ for all $\mathbf{M}$-evaluations $v$. We say that $\langle\boldsymbol{A},\mathbf{M}\rangle$ is a \emph{model of a set of formulas $\Phi$}, if $\|\varphi \|^{\boldsymbol{A}}_{\mathbf{M}}=1$ for any $\varphi\in\Phi$. Sometimes we will denote by $\overrightarrow{x}$ a sequence of variables $x_1,\ldots,x_n$ (and the same with sequences $\overrightarrow{d}$ of elements of the domain). Given a structure $\langle\boldsymbol{A},\mathbf{M}\rangle$ and a formula $\varphi(\overrightarrow{x})$, we say that $\overrightarrow{d}\subseteq M$ \emph{satisfies} $\varphi(\overrightarrow{x})$ (or that $\varphi(\overrightarrow{x})$ is \emph{satisfied} by $\overrightarrow{d}$) if $\semvalue{\varphi(\overrightarrow{x})}^{\emph{\textbf{A}}}_{\textbf{M},v[\overrightarrow{x}\rightarrow \overrightarrow{d}]}=\1^\alg{A}$ for any $\textbf{M}$-evaluation $v$ (also written $\semvalue{\varphi [\overrightarrow{d}]}^{\emph{\textbf{A}}}_{\textbf{M}}=\1^\alg{A}$); for the sake of clarity we will use also the notation $\tuple{\alg{A}, \struct{M}}\models \f[\overrightarrow{d}]$ when is needed.
Two theories $T$ and $U$ are said to be \emph{1-equivalent} if a structure is  a model of $T$ if it is also a model of $U$ (in the case where $T$ and $U$ are singletons of formulas, we say that these formulas are 1-equivalent).

Given a set of sentences $\Sigma$, and a sentence $\phi$, we denote by $\Sigma \models_{\boldsymbol{A}}\phi$ the fact that every $\boldsymbol{A}$-model of $\Sigma$ is also an $\boldsymbol{A}$-model of $\phi$. We focus on classes of structures over a fixed finite chain \emph{\textbf{A}} whose set of elements is denoted by $\{a_1,\ldots,a_k\}$. Such restriction is due to the fact that dropping finiteness can cause to lose compactness, which is an essential element of our proofs. However, the results will still be quite encompassing in practice.  Indeed, for instance, prominent examples of weighted structures in computer science are valued over finite chains. Structures over a fixed finite chain \emph{\textbf{A}} have two important properties: they are witnessed (the values of the quantifiers are maxima and minima achieved in particular instances), and have the compactness property, both for satisfiability and for consequence (see e.g.~\cite{De14}). 
\begin{Pro}  Let $\boldsymbol{A}$ be a fixed finite chain. For every set of sentences $\Sigma \cup \{\alpha\}$, the following holds: 
\begin{enumerate}
\item If every finite subset $\Sigma_0\subseteq\Sigma$ has a model $\langle\textbf{A},\emph{\textbf{M}}_{\Sigma_0}\rangle$, then $\Sigma$ has a model $\langle\textbf{A},\emph{\textbf{N}}\rangle$. 
\item If $\Sigma \models_{\boldsymbol{A}} \alpha$, then there is a finite subset $\Sigma_0\subseteq\Sigma$ such that $\Sigma_0 \models_{\boldsymbol{A}} \alpha$.
\end{enumerate}
\end{Pro}

From now on we refer to \emph{\textbf{A}}-structures simply as \emph{structures} (or as {\em $\mathcal{P}$-structures} if we need to specify the language). For the remainder of the article, let us assume that we have a crisp identity $\approx$ in the language. 

\begin{Def} \label{def:mapping structures}
$\space$ Let $\mathcal{P}$ be a predicate language, $\langle\textbf{A},\mathrm{\mathbf{M}}\rangle$ and $\langle\textbf{B},\mathrm{\mathbf{N}}\rangle$ structures for $\pl$, $f$ a mapping from $\textbf{A}$ to $\textbf{B}$ and $g$ a mapping from $M$ to $N$. The pair $\langle f,g\rangle$ is said to be a \emph{strong homomorphism} from $\langle\textbf{A},\mathrm{\mathbf{M}}\rangle$ to $\langle\textbf{B},\mathrm{\mathbf{N}}\rangle$ if $f$ is an algebraic homomorphism and for every $n$-ary function symbol $F\in\mathcal{P}$ and $d_1,\ldots,d_n\in M$,
$$g(F_{\mathrm{\mathbf{M}}}(d_1,\ldots,d_n))=F_{\mathrm{\mathbf{N}}}(g(d_1),\ldots,g(d_n)) $$

\noindent and for every $n$-ary predicate symbol $P\in\mathcal{P}$ and \newline $d_1,\ldots,d_n\in M,$ \begin{center}

$ f(\semvalue{P(d_1\ldots,d_n)}^{\textbf{A}}_{{\mathrm{\mathbf{M}}}})= \semvalue{P(g(d_1),\ldots,g(d_n))}^{\textbf{B}}_{{\mathrm{\mathbf{N}}}}.$
\end{center}

A strong homomorphism $\langle f,g\rangle$ is said to be \emph{elementary}  if we have, for every $\pl$-formula $\varphi(x_1,\ldots,x_n)$ and $d_1,\ldots,d_n\in M$,
$$  f(\semvalue{\varphi(d_1\ldots,d_n)}^{\textbf{A}}_{\mathrm{\mathbf{M}}})=\semvalue{\f (g(d_1),\ldots,g(d_n)}^{\textbf{B}}_{\mathrm{\mathbf{N}}}.$$
\end{Def}  

Let $\langle f,g\rangle$ be a strong homomorphism from $\langle\emph{\textbf{A}},\mathrm{\mathbf{M}}\rangle$ to $\langle\emph{\textbf{B}},\mathrm{\mathbf{N}}\rangle$, we say that $\langle f,g\rangle$ is an \emph{embedding} from $\langle\emph{\textbf{A}},\mathrm{\mathbf{M}}\rangle$ to $\langle\emph{\textbf{B}},\mathrm{\mathbf{N}}\rangle$ if both functions $f$ and $g$ are injective, and we say that $\langle f,g\rangle$ is an \emph{isomorphism} from $\langle\emph{\textbf{A}},\mathrm{\mathbf{M}}\rangle$ to $\langle\emph{\textbf{B}},\mathrm{\mathbf{N}}\rangle$ if $\langle f,g\rangle$ is an embedding and both functions $f$ and $g$ are surjective. For a general study of different kinds of homomorphisms and the formulas they preserve we refer to \cite{DeGaNo16}.

Later in the article we will use diagram techniques. We present here some corollaries of the results obtained in \cite{De11}. Given a language $\mathcal{P}$, we start by introducing three different expansions adding either a new truth-constant for each elements of the algebra, or new object constants. For any element $a$ of $\alg{A}$, we will use the truth-constant $\overline{a}$ to denote it. When $a=\1^\alg{A}$ or $a=\0^\alg{A}$, then $\overline{a}=\1$ or $\overline{a}=\0$, respectively.

\begin{Def} \label{lalg}  Given a predicate language $\mathcal{P}$, we expand it by adding an individual constant symbol $c_m$ for every $m\in M$, and denote it by $\mathcal{P^{\mathbf{M}}}$. If $\langle\textbf{A},\mathrm{\mathbf{M}}\rangle$ is a $\mathcal{P^{\mathbf{M}}}$-structure, we denote by $\langle\textbf{A},\mathrm{\mathbf{M^\sharp}}\rangle$ the expansion of the structure $\langle\textbf{A},\mathrm{\mathbf{M}}\rangle$ to  $\mathcal{P^{\mathbf{M}}}$, where for every $m\in M$, $(c_m)_\mathrm{\mathbf{\emph{\textbf{M}}^\sharp}}=m$.
\end{Def} 

\begin{Def} \label{lalg} Given a predicate language $\mathcal{P}$, we expand it by adding a truth constant symbol $\overline{a}$ for every $a \in A$, and denote it by $\mathcal{P^{\textbf{A}}}$. When we expand the language $\mathcal{P^{\textbf{A}}}$ further by adding an individual constant symbol $c_m$ for every $m\in M$, we will denote it by $\mathcal{P^{\langle\textbf{A},\mathrm{\mathbf{M}}\rangle}}$. 
\end{Def} 

\begin{Def} \label{ElDiag} Let $\mathcal{P}$ be a predicate language and $\langle\textbf{A},\mathrm{\mathbf{M}}\rangle$ a $\mathcal{P}$-structure.  We define the following sets of $\mathcal{P^{\langle\textbf{A},\mathrm{\mathbf{M}}\rangle}}$-sentences:
$$ \emph{ElDiag}(\textbf{A},\emph{\textbf{M}})=\{ \sigma\leftrightarrow \overline{a} \mid \sigma \text{ is a sentence of } \mathcal{P^{\mathbf{M}}},  a\in A$$ 
$$\mbox{and } \ \semvalue{\sigma}^{\textbf{A}}_{\mathbf{M^{\sharp}}}=a\},$$
whereas $\emph{Diag}(\textbf{A},\emph{\textbf{M}})$ is the subset of\/ $\emph{ElDiag}(\textbf{A},\emph{\textbf{M}})$ containing all formulas $\sigma\leftrightarrow \overline{a}$ where $\sigma$ is quantifier-free.
\end{Def} 

Following the same lines of the proof of \cite[Prop. 32]{De11}, we can obtain a characterization of strong and elementary embeddings between two $\mathcal{P}$-structures over a chain $\alg{A}$.

\begin{Cor}\label{lemma cor} Let $\langle\textbf{A},\mathrm{\mathbf{M}}\rangle$ and $\langle\textbf{A},\mathrm{\mathbf{N}}\rangle$ be two $\mathcal{P}$-structures for $\mathcal{P^{\textbf{A}}}$. The following are equivalent:
\begin{enumerate}
\item There is an expansion of $\langle\textbf{A},\mathrm{\mathbf{N}}\rangle$ that is a model of \emph{Diag}$(\textbf{A},\emph{\textbf{M}})$ (\emph{ElDiag}$(\textbf{A},\emph{\textbf{M}})$, respectively).
\item There is a mapping $g\colon M \to N$ such that $\langle Id_\textbf{A},g \rangle$ is a strong (elementary, respectively) embedding from $\langle\textbf{A},\mathrm{\mathbf{M}}\rangle$ into $\langle\textbf{A},\mathrm{\mathbf{N}}\rangle$.
\end{enumerate}
\end{Cor}

\section{Universal Classes}\label{s:universal}

In this section we prove a result on existential amalgamation (Proposition~\ref{am}) from which we extract a \L o\'s--Tarski preservation theorem for universal theories  (Theorem~\ref{thm semantic interpolation}) and a characterization of universal classes of structures (Theorem~\ref{thm characterization universal classes}). Relevant structures in computer science are axiomatized by sets of universal formulas; one prominent example is the class of weighted graphs. Particular versions of the above mentioned results appeared for \L$\forall$ in~\cite{Spa09}. In the context of fuzzy logic programming, Gerla~\cite{Ge05} studied universal formulas with relation to Herbrand interpretations.

%
%
%
%
%


For the upcoming results, we need to recall the notion of substructure.

\begin{Def}\emph{(Substructure)}\label{def cintula}
Let $\langle\textbf{A},\emph{\textbf{M}}\rangle$ and $\langle\textbf{B},\emph{\textbf{N}}\rangle$ be $\mathcal{P}$-structures. We say that $\langle\textbf{A},\emph{\textbf{M}}\rangle$ is a \emph{substructure} of $\langle\textbf{B},\emph{\textbf{N}}\rangle$ if:
\begin{itemize}
\item[(1)] $\textbf{A}$ is a subalgebra of $\textbf{B};$
\item[(2)] $M\subseteq N;$
\item[(3)] for any $n$-ary function symbol $F\in\mathcal{P}$ and elements $d_1,\ldots,d_n\in M$, we have 
\begin{center}
$F_{\emph{\textbf{M}}}(d_1,\ldots,d_n)=F_{\emph{\textbf{N}}}(d_1,\ldots,d_n);$
\end{center}
\item[(4)] for any $n$-ary predicate symbol $P\in\mathcal{P}$ and elements $d_1,\ldots,d_n\in M$,  we have
\begin{center}
$P_{\emph{\textbf{M}}}(d_1,\ldots,d_n)=P_{\emph{\textbf{N}}}(d_1,\ldots,d_n)$.
\end{center}
\end{itemize}
\end{Def} 

Remark that $\langle\emph{\textbf{A}}, \textbf{M}\rangle$ is a substructure of $\langle\emph{\textbf{B}},\textbf{N}\rangle$ if and only if conditions (1)-(3) are satisfied and, instead of (4), the following condition holds: for every quantifier-free formula $\varphi(x_1,\ldots,x_n)$ and any elements $d_1,\ldots,d_n\in M$,
$$ \semvalue{\varphi(d_1,\ldots,d_n)}^{\emph{\textbf{A}}}_{\textbf{M}}=\semvalue{\varphi(d_1,\ldots,d_n)}^{\emph{\textbf{B}}}_{\textbf{N}}.$$ 

With this notion at hand, we can define a corresponding closure property for classes of structures.

\begin{Def}\emph{(Class Closed Under Substructures)}
Let $\K$ be a class of $\mathcal{P}$-structures. We say that $\K$ is \emph{closed under substructures} if, for any structure $\langle\textbf{A},\emph{\textbf{M}}\rangle\in \K$, 
$$ \text{if } \langle\textbf{B},\emph{\textbf{N}}\rangle \text{ is a substructure of } \langle\textbf{A},\emph{\textbf{M}}\rangle\text{, then } \langle\textbf{B},\emph{\textbf{N}}\rangle\in\K.$$
\end{Def} 

Since our characterizations will be based on axiomatizability of classes, we need to recall the definition of elementary class of structures.

\begin{Def} \emph{(Elementary Class \cite[Def. 2.15]{BuSan81})} \label{def class axiomatized}
A class $\K$ of $\mathcal{P}$-structures is an \emph{elementary class} (or a \emph{first-order class}) if there is a set $\Sigma$ of sentences such that for every $  \langle\textbf{A},\emph{\textbf{M}}\rangle$,
$$  \langle\textbf{A},\emph{\textbf{M}}\rangle\in \K \text{ if and only if } \text{    }\tuple{\alg{A}, \struct{M}}\models \Sigma.  $$
In this case, $\K$ is said to be \emph{axiomatized} (or \emph{defined}) \emph{by} $\Sigma$. 
\end{Def} 

Using a predicate language with only one binary relation $R$,
the class of weighted undirected graphs is axiomatized by the following set of universal sentences: 
$$\{(\forall x)(R(x,x)\rightarrow\overline{0}),(\forall x)(\forall y)(R(x,y)\rightarrow R(y,x))\}.$$

Notice that the notion of induced weighted undirected subgraph corresponds to the model-theoretic notion of substructure used in MFL.


\begin{Def}
Let $\mathcal{P}$ be a predicate language. We say that a $\mathcal{P}$-formula $\varphi(x_1,\ldots,x_n)$ is \emph{preserved under substructures} if for any $\mathcal{P}$-structure $\langle\textbf{A},\emph{\textbf{M}}\rangle$ and any substructure $\langle\textbf{B},\emph{\textbf{N}}\rangle$, if $\semvalue{\varphi(d_1,\ldots,d_n)}^{\textbf{A}}_{\emph{\textbf{M}}}=\1^\alg{A}$ for some  $d_1,\ldots,d_n\in N$, then $\semvalue{\varphi(d_1,\ldots,d_n)}^{\textbf{B}}_{\emph{\textbf{N}}}=\1^\alg{B}$.
\end{Def} 

The following lemma can be easily proved by induction on the complexity of universal formulas. 

\begin{Lem} \label{pres} Let $\varphi(x_1,\ldots,x_n)$ be a universal formula. Then, $\varphi(x_1,\ldots,x_n)$ is preserved under substructures. 
\end{Lem}

In classical model theory amalgamation properties are often related in elegant ways to preservation theorems (see e.g.~\cite{h}). We will try an analogous approach to obtain our desired preservation result. The importance of this idea is that the problem of proving a preservation result reduces then to finding a suitable amalgamation counterpart. This provides us with proofs that have a neat common structure (such as those of the main results in this section and the next one).

We will write $ \tuple{\textbf{A}, {\bf M}_2, \overrightarrow{d}}   \Rrightarrow_{\exists_n}  \tuple{\textbf{A}, {\bf M}_1, \overrightarrow{d}} $ if for any $\exists_n$-formula $\f$, $ \tuple{\textbf{A}, {\bf M}_2} \models \f [\overrightarrow{d}] $ only if $ \tuple{\textbf{A}, {\bf M}_1} \models \f [\overrightarrow{d}]$.

\begin{Pro}\label{am}\emph{(Existential amalgamation)}
 Let $\tuple{\textbf{A}, {\bf M}_1}$ and $\tuple{\textbf{A}, {\bf M}_2}$ be two structures for $\mathcal{P}^{\textbf{A}}$ with a common part $\tuple{\textbf{A}, {\bf M}}$ with domain generated by a sequence of elements $\overrightarrow{d}$. Moreover, suppose that  $$   \tuple{\textbf{A}, {\bf M}_2, \overrightarrow{d}} \Rrightarrow_{\exists_1} \tuple{\textbf{A}, {\bf M}_1, \overrightarrow{d}}. $$ Then there is a structure $\tuple{\textbf{A}, {\bf N}}$ into which $\tuple{\textbf{A}, {\bf M}_2}$ can be strongly embedded by $\tuple{f, g}$ while $\tuple{\textbf{A}, {\bf M}_1}$ is $\mathcal{P}^{\textbf{A}}$-elementarily strongly embedded (taking isomorphic copies, we may assume that $\tuple{\textbf{A}, {\bf M}_1}$ is just a  $\mathcal{P}^{\textbf{A}}$-elementary substructure).  The situation is described by the following picture:
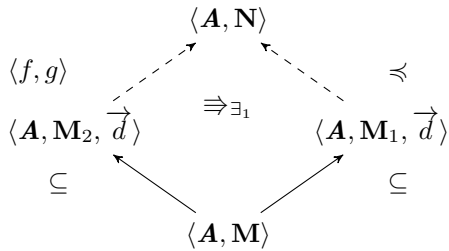
\begin{center}
\begin{tikzpicture}
    \node (E) at (0,0) {};
    \node[above=of E] (F) {$\tuple{\textbf{A}, {\bf N}}$};
        \node[below=of E] (C) {$\tuple{\textbf{A}, {\bf M}}$};

 \node[midway, below, left=of E] (A) {$\tuple{\textbf{A}, {\bf M}_2, \overrightarrow{d}}$};
  \node[midway, below, right=of E] (B) {$\tuple{\textbf{A}, {\bf M}_1, \overrightarrow{d}}$};
    \node[below=of E] (Asubt) {};
  
    \draw[->, dashed] (A)--(F) node [midway, below,  left=of E] {$\tuple{f, g}$ };
\draw[->, dashed] (B)--(F) node [midway, below, right=of E] {  $\preccurlyeq$};
\draw  (B) (A) node [midway,above] {$\Rrightarrow_{\exists_1}$};
\draw[->] (C)--(B) node [midway, below, right=of E] {  $\subseteq$};
\draw[->] (C)--(A) node [midway, below, left=of E] {  $\subseteq$};
\end{tikzpicture}
\end{center}
Moreover, the result is also true when $\tuple{\textbf{A}, {\bf M}_1}$ and $\tuple{\textbf{A}, {\bf M}_2}$ have no common part.
\end{Pro}

\begin{proof} It is not a difficult to show that $  \mbox{ElDiag}(\emph{\textbf{A}}, {\bf M}_1) \cup \mbox{Diag}(\emph{\textbf{A}}, {\bf M}_2)$ (where we let the elements of the domain serve as constants to name themselves) has a model, which suffices for the purposes of the result.  Suppose otherwise, that is, for some finite $\mbox{Diag}_0(\emph{\textbf{A}}, {\bf M}_2) \subseteq \mbox{Diag}(\emph{\textbf{A}}, {\bf M}_2)$, we have that $$ \mbox{ElDiag}(\emph{\textbf{A}}, {\bf M}_1) \vDash (\bigwedge  \mbox{Diag}_0(\emph{\textbf{A}}, {\bf M}_2))  \rightarrow  \overline{a}$$ for the immediate predeccessor  $a $ in the lattice order of \emph{\textbf{A}}  of $\1^\emph{\textbf{A}}$.

Quantifying away the new individual constants, we obtain a set of formulas $ \mbox{Diag}_0^{*}(\emph{\textbf{A}}, {\bf M}_2)$ such that:
 $$ \mbox{ElDiag}(\emph{\textbf{A}}, {\bf M}_1) \vDash (\exists \overrightarrow{x}) ((\bigwedge  \mbox{Diag}_0^{*}(\emph{\textbf{A}}, {\bf M}_2)) ) \rightarrow  \overline{a}.$$

Since $\tuple{\emph{\textbf{A}}, {\bf M}_2, \overrightarrow{d}}   \Rrightarrow_{\exists_1}  \tuple{\emph{\textbf{A}}, {\bf M}_1, \overrightarrow{d}}$,
then  $$ \tuple{\emph{\textbf{A}}, {\bf M}_2} \not \models (\exists \overrightarrow{x}) (\bigwedge  \mbox{Diag}_0^{*}(\emph{\textbf{A}}, {\bf M}_2)),$$ which is a contradiction. 
Note, moreover, that if  $$ \tuple{\emph{\textbf{A}}, {\bf M}_2, \overrightarrow{d}}   \Rrightarrow_{\exists_1}  \tuple{\emph{\textbf{A}}, {\bf M}_1, \overrightarrow{d}}, $$ we also have that whenever $\f(\bar{x}) $ is quantifier-free formula of $\pl^{\alg{A}}$, $ \tuple{\emph{\textbf{A}}, {\bf M}_2 } \models \f[\overrightarrow{d}]$ iff $ \tuple{\emph{\textbf{A}}, {\bf M}_1 } \models \f[\overrightarrow{d}] $. Left-to-right is clear; the contrapositive of the right-to-left direction follows easily: if $\tuple{\emph{\textbf{A}}, {\bf M}_1 } \not \models \f[\overrightarrow{d}]$, then $\tuple{\emph{\textbf{A}}, {\bf M}_1 } \not \models \f \leftrightarrow \overline{a} [\overrightarrow{d}]$ for some $a \neq \1^{\alg{A}}$, so $ \tuple{\emph{\textbf{A}}, {\bf M}_2 } \models \f \leftrightarrow  \overline{a}[\overrightarrow{d}]$, which means that $ \tuple{\emph{\textbf{A}}, {\bf M}_2 } \not \models \f [\overrightarrow{d}]$.

Observe that the proof can be similarly carried out, \emph{mutatis mutandi}, when $\tuple{\emph{\textbf{A}}, {\bf M}_1}$ and $\tuple{\emph{\textbf{A}}, {\bf M}_2}$ have no common part as well.\qed\end{proof}

Now we have the elements to establish an exact analogue of Theorem 5 from~\cite{l}, \L o\'s--Tarski preservation theorem. 

\begin{Thm} \emph{(\L o\'s--Tarski preservation theorem)} \label{thm semantic interpolation} Let $T$ be a $\pl^{\alg{A}}$-theory and $\Phi(\overrightarrow{x})$ a set of formulas in $\pl^{\alg{A}}$. Then the following are equivalent:
\begin{itemize}
\item [(i)] For any models of T, $\tuple{\alg{A}, \struct{M}} \subseteq \tuple{\alg{A}, \struct{N}}$, we have:\\ if $\tuple{\alg{A}, \struct{N}} \models \Phi$, then $\tuple{\alg{A}, \struct{M}} \models \Phi$.
\item [(ii)] There is a set of universal $\pl^{\alg{A}}$-formulas $\Theta(\overrightarrow{x})$ such that: $T, \Phi \vDash  \Theta$ and  $T,  \Theta  \vDash \Phi$.
\end{itemize}
\end{Thm}

\begin{proof} Let us prove the difficult direction (the converse direction is clear by Lemma~\ref{pres}).  Consider $(T \cup \Phi(\overrightarrow{x}))_{\forall_1}$, the collection of all $\forall_1$ logical consequences of $T \cup \Phi(\overrightarrow{x})$. We need to establish that the only models of  $(T \cup \Phi(\overrightarrow{x}))_{\forall_1}$ among the models of $T$ are the substructures of models of $\Phi(\overrightarrow{x})$. Let $\tuple{\alg{A}, \struct{M}}$ be a model of  $(T \cup \Phi(\overrightarrow{x}))_{\forall_1}$. All we need to do is find a model $\tuple{\alg{A}, \struct{N}}$ of the theory $T \cup \Phi(\overrightarrow{x})$ such that  $\tuple{\alg{A}, \struct{M}} \Rrightarrow_{\exists_1} \tuple{\alg{A}, \struct{N}}$ and then quote the existential amalgamation theorem.

Let $U$ be all $\exists_1$-formulas that hold in in $\tuple{\alg{A}, \struct{M}}$. We claim  then   that $T \cup \Phi(\overrightarrow{x}) \cup U$ has a model. Otherwise, by compactness, for $$\{(\exists \overrightarrow{x_0}) \phi_0(\overrightarrow{x_0}), \dots,  (\exists \overrightarrow{x_n}) \phi_0(\overrightarrow{x_n})\} \subseteq U$$ we have that in all models of $T \cup \Phi(\overrightarrow{x})$, it holds that $$ (\exists \overrightarrow{x}_0) \phi_0(\overrightarrow{x_0}) \wedge \dots \wedge (\exists \overrightarrow{x_n}) \phi_0(\overrightarrow{x_n}) \rightarrow \overline{a},$$ where $a$ is the immediate predecessor of $\1^\alg{A}$, and by basic manipulations, $$ (\exists \overrightarrow{x_0}, \dots, \overrightarrow{x_n}) (\phi_0(\overrightarrow{x_0}) \wedge \dots \wedge \phi_0(\overrightarrow{x_n})) \rightarrow \overline{a},$$ which is just equivalent to  $$ (\forall \overrightarrow{x_0}, \dots, \overrightarrow{x_n}) (\phi_0(\overrightarrow{x_0}) \wedge \dots \wedge \phi_0(\overrightarrow{x_n}) \rightarrow \overline{a}).$$ 
The latter formula must be in $(T \cup \Phi(\overrightarrow{x}))_{\forall_1}$ then, which is a contradiction.\qed
\end{proof}


Following a similar proof, we can obtain an algebraic characterization equivalent to  Theorem~\ref{thm semantic interpolation}.
\begin{Thm} \label{thm characterization universal classes} Let $\K$ be a class of $\mathcal{P^{\textbf{A}}}$-structures. Then, the following are equivalent: 

\begin{itemize}
\item[(i)] $\K$ is closed under  isomorphisms, substructures, and ultraproducts.

\item[(ii)]  $\K$ is axiomatized by a set of universal $\mathcal{P^{\textbf{A}}}$-sentences.
\end{itemize}
\end{Thm}

The following corollary can be obtained because in our setting two forms of compactness (that are generally distinct, in, say, \L ukasiewicz logic) collapse, namely (1) the compactness of the consequence relation and (2) the compactness of the satisfiability relation. (1) clearly implies (2) in the presence of \0 in our language. To see the converse, say that $T \vDash \f$, which amounts to say that $T \cup \{\f \rightarrow \overline{a}\}$ (where $a$ is the predecessor of $\1^\alg{A}$) does not have a model. Hence, by (2), there is a finite $T_0 \subseteq T$ such that $T_0 \cup \{\f \rightarrow \overline{a}\}$ has no model, so, in fact, $T_0 \vDash \f$. 

\begin{Cor} \label{tarski} Let $T \cup \{\f\}$ be a set of $\pl^{\alg{A}}$-sentences. Then, $\f$ is preserved under substructures of models of $T$ if, and only if, $\f$ is 1-equivalent to a universal $\pl^{\alg{A}}$-sentence modulo $T$.
\end{Cor}

\begin{proof} Apply Theorem~\ref{thm semantic interpolation} for $\Phi = \{\f\}$. Consequently, $\f$ is axiomatized by a set of universal $\pl^{\alg{A}}$-sentences. Then bring it down to a single such formula using \emph{\textbf{A}}-compactness for consequence.\qed\end{proof}

A natural question is whether Corollary~\ref{tarski} can be strengthened to strong equivalence in terms of $\leftrightarrow$, that is, whether one can find a universal formula that agrees with $\phi$ on each value in every structure (not just on value $\1^{\alg{A}}$). Following the lines of the above proof this would require to show something like, for an arbitrary model $\tuple{\alg{A}, \struct{M}}$,
$$
\semvalue{\psi}^{\emph{\textbf{A}}}_{\textbf{M}} \leq^{\boldsymbol{A}} \semvalue{\phi}^{\emph{\textbf{A}}}_{\textbf{M}} \,\, (\mbox{for all $\psi$ s.t.} \vDash \phi \rightarrow \psi) .
$$
Then, one would expect to reduce the left side of the inequality to a finite set $\Psi$ such that
$$
inf_{\leq^{\boldsymbol{A}}} \{\semvalue{\psi}^{\emph{\textbf{A}}}_{\textbf{M}} \mid  \psi \in \Psi\} \leq^{\boldsymbol{A}} \semvalue{\phi}^{\emph{\textbf{A}}}_{\textbf{M}}.
$$
However, this reduction would come from compactness in the usual argument, but it does not in this one. This is because compactness is about consequence as opposed to implication, which are different in a setting without a deduction theorem such as this. In fact, in~\cite{Spa09} similar results in the framework \L ukasiewicz logic are obtained only for $1$-equivalence as well.

Needless to say, the previous results, in particular, allow to conclude that a class of $\mathcal{P}$-structures (that is, structures for a language without additional truth-constants) closed under substructures can be axiomatized by universal $\mathcal{P^{\textbf{A}}}$-sentences. One might wonder, of course, if it is really necessary to resort a universal axiomatization in the expanded language.

Let us present a counterexample showing that, in general, the base language $\mathcal{P}$ does not suffice. Let $\mathcal{P}$ be the language with only one monadic predicate $P$ and take two structures over the standard G\"odel chain, $\langle [0,1]_\G, \struct{M}\rangle$ and $\langle [0,1]_\G, \struct{N} \rangle$. The domain in both cases is the set of all natural numbers $\mathbb{N}$ and the interpretation of the predicate is respectively defined as: $ P_{\struct{M}} (n)=\frac{3}{4} $, and $ P_{\struct{N}} (n)=\frac{1}{2}$, for every $n \in \mathbb{N}$. First we show that $\langle [0,1]_\G, \struct{M}\rangle \equiv \langle [0,1]_\G, \struct{N} \rangle$. Take $f$ as any non-decreasing bijection from $[0,1]$ to $[0,1]$ such that $f(\frac{3}{4})=\frac{1}{2}$, $f(1) = 1$, $f(0) = 0$. It is easy to check that $f$ is a $\G$-homomorphism preserving suprema and infima. Then, we can consider the $\sigma$-mapping $\tuple{f,\mathit{Id}}$ and apply~\cite[Lemma 11]{DeGaNo18} to obtain that $\langle [0,1]_\G, \struct{M}\rangle \equiv \langle [0,1]_\G, \struct{N} \rangle$. Consider now the finite subalgebra $\emph{\textbf{A}}$ of $[0,1]_\G$ generated by the subset $\{0,\frac{1}{2},\frac{3}{4}, 1\}$. Clearly, the structures $\langle [0,1]_\G, \struct{M}\rangle$ and $\langle [0,1]_\G, \struct{N} \rangle$ can be regarded as structures over $\emph{\textbf{A}}$. Thus, we have $$\langle\emph{\textbf{A}}, \struct{M}\rangle \equiv \langle \emph{\textbf{A}}, \struct{N} \rangle.$$
Observe that $\semvalue{(\forall x) P(x)}^{\emph{\textbf{A}}}_{\struct{M}} =\frac{3}{4}$ and $\semvalue{(\forall x) P(x)}^{\emph{\textbf{A}}}_{\struct{N}} = \frac{1}{2}$. Consider the expanded language $\mathcal{P}_{\emph{\textbf{A}}}$ obtained by adding a constant symbol $\overline{a}$ for every element $a \in A$. Let $ \K$ be the class of $\pl$-structures valued on $\alg{A}$, whose natural expansion to $\mathcal{P}_{\emph{\textbf{A}}}$ (that is, the expansion in which every constant $\overline{a}$ is interpreted as the corresponding element $a$) satisfies the sentence $$\overline{\frac{3}{4}} \to (\forall x) P(x).$$ Clearly, $\K$ is closed under substructures and $\langle\emph{\textbf{A}}, \struct{M}\rangle\in \K$. However $\langle\emph{\textbf{A}}, \struct{N}\rangle\notin \K$, because $\semvalue{(\forall x) P(x)}^{\emph{\textbf{A}}}_{\struct{N}} = \frac{1}{2}$. Therefore, $ \K$ cannot be axiomatized by a set of universal $\mathcal{P}$-sentences, since it contains $\langle\emph{\textbf{A}}, \struct{M}\rangle$ but not the elementary equivalent $\langle \emph{\textbf{A}}, \struct{N} \rangle$. Hence, we have produced an example of a class of $\pl$-structures closed under substructures (and, obviously, under isomorphisms and ultraproducts) which is not axiomatizable with universal $\pl$-sentences.

\section{Universal-existential classes}\label{s:universal-existential}

This section runs quite parallel to the previous one. We recall the notion of elementary chain of structures and its corresponding Tarski--Vaught theorem and prove that universal-existential formulas are preserved under unions of chains (Lemma~\ref{l:preun1}). After that, we prove a result on existential-universal amalgamation (Proposition~\ref{am2}) and derive from it a Chang--\L o\'s--Suszko preservation theorem (Theorem~\ref{Chang-Los-Suszko}).

Consider the class $\K$ of all structures in a signature with a binary function symbol $\cdot$, a unary function symbols $^{-1}$, an individual constant $1$, and a unary predicate $G$ satisfying the following axioms: 
\begin{enumerate}

\item[] $(\forall x) (\exists y) (y^n \approx x)$ for each  $n \geqslant 2$.
\item[] $(\forall x, y) ((x\cdot y)\cdot z \approx x\cdot(y \cdot z))$
\item[] $(\forall x) (x \cdot 1 \approx x)$
\item[] $(\forall x) (x\cdot x^{-1} \approx 1)$
\item[] $(\forall x, y)(x\cdot y \approx y\cdot x)$
\item[] $(\forall x,y) ((Gx \wedge Gy) \rightarrow G(xy))$
\item[] $(\forall x) (Gx \rightarrow G(x^{-1}))$
\end{enumerate}

This is the class of divisible Abelian groups with a fuzzy subgroup defined by the predicate $G$ (following the definition of~\cite{Ros71}). By our Chang--\L o\'s--Suszko preservation theorem below,  $\K$ is a class closed under unions of chains.  

Another example of such class be provided by the class of all weighted graphs where the formula $$(\forall x) (\exists y, z) (y \not\approx z \wedge Rxy \wedge Rxz)$$ holds, that is, every  vertex has  at least two incident edges. This axiomatizes the class of graphs where every vertex has degree $\geqslant 2$.

Given an ordinal $\gamma$, a sequence $\{\tuple{\alg{ A}, {\bf M}_i} \mid i < \gamma\}$ of models is called a \emph{chain} when for all $i<j<\gamma$ we have that $\tuple{\alg{ A} , {\bf M}_i}$ is a substructure of  $\tuple{\alg{ A}, {\bf M}_j}$. If, moreover, these substructures are elementary, we speak of an \emph{elementary chain}. The \emph{union} of the chain $\{\tuple{\alg{ A}, {\bf M}_i} \mid i < \gamma\}$ is the structure $\tuple{\alg{ A}, {\bf M}}$ where  $\struct{ M} $ is defined by taking as its domain $\bigcup _{i<\gamma}{ M}_i$, interpreting the constants of the language as they were interpreted in each ${\bf M}_i$ and similarly with the relational symbols of the language. Observe as well that ${\bf M}$ is well defined given that $\{\tuple{\alg{ A}, {\bf M}_i} \mid i < \gamma\}$ is a chain.

Next we recall a useful theorem that has been established and used to construct saturated models in the context of mathematical fuzzy logic in~\cite{BN18}.

\begin{Thm}[\hspace{-0.04cm}\cite{BN18}]\emph{(Tarski--Vaught)}\label{t:Unions} Let $\tuple{\alg{A}, \struct{M}}$ be the union of the elementary chain $\{\tuple{\alg{A},\struct{M}_i} \mid i < \gamma\}$. Then, for every sequence $\overrightarrow{d}$ of elements of\/ ${\bf M}_i$ and formula $\f$, $ \semvalue{\f (\overrightarrow{d})}^\alg{A}_{\struct{M}} = \semvalue{\f (\overrightarrow{d})}^\model{\alg{A}}_{\struct{M}_i}$. Moreover, if the chain is not elementary, we still have that $ \semvalue{\f (\overrightarrow{d})}^\alg{A}_{\struct{M}} = \semvalue{\f (\overrightarrow{d})}^\model{\alg{A}}_{\struct{M}_i}$ for every quantifier free formula.
\end{Thm}

Therefore, unions of elementary chains preserve the values of all formulas. It is also interesting to consider formulas that are preserved by {\em all} unions of chains.

\begin{Def} We say that a formula $\varphi(x_1,\ldots,x_n)$ is \emph{preserved under unions of chains} if whenever we have a chain of models $\{\tuple{\alg{A}, {\bf M}_i} \mid i < \gamma\}$   such that for every $i$, $\semvalue{\varphi(\overrightarrow{d})}^\alg{A}_{\struct{M}_i} = \1^\alg{A} ( i < \gamma)$ for some sequence $\overrightarrow{d}$ of elements of $M_0$, then $\semvalue{\varphi(\overrightarrow{d})}^\alg{A}_{\struct{M}} = \1^\alg{A}$, where $\tuple{\alg{ A}, {\bf M}}$ is the union of the chain.
\end{Def}

Let $a$ be the element of $\alg{A}$ immediately above $\0^\alg{A}$.

\begin{Lem}\label{l:preun1}
$\forall_2$-formulas are preserved under unions of chains.
\end{Lem}

\begin{proof} Let $ (\forall \overrightarrow{x})(\exists \overrightarrow{y}) \phi$ be a $\forall_2$-formula, $\langle\emph{\textbf{A}},\mathrm{\mathbf{M}}\rangle$ be the union of a chain $\{\langle\emph{\textbf{A}},\mathrm{\mathbf{M}_i}\rangle \mid i < \gamma \}$, and $\overrightarrow{c}$ some sequence of elements of $M_0$. Assume that for every $i < \gamma$, $ \semvalue{(\forall \overrightarrow{x})(\exists \overrightarrow{y}) \phi(\overrightarrow{c})}^\alg{A}_{\mathbf{M}_i} = \1^\alg{A}$. Let $\overrightarrow{d} \in M$, we show that

$ \semvalue{(\exists \overrightarrow{y}) \phi(\overrightarrow{d}, \overrightarrow{c})}^\alg{A}_{\mathbf{M}} = \1^\alg{A}$. Let $j< \gamma$ be such that $\overrightarrow{d} \in M_j$. Since $ \semvalue{(\forall \overrightarrow{x})(\exists \overrightarrow{y}) \phi(\overrightarrow{c})}^\alg{A}_{\mathbf{M}_j} = \1^\alg{A}$ we have $\semvalue{(\exists \overrightarrow{y}) \phi(\overrightarrow{d},\overrightarrow{c})}^\alg{A}_{\mathbf{M}_j} = \1^\alg{A}$. Since $\langle\emph{\textbf{A}},\mathrm{\mathbf{M}_j}\rangle$ is $\exists$-witnessed, there are $\overrightarrow{e} \in M_j$ such that $\semvalue{ \phi(\overrightarrow{d},\overrightarrow{e},\overrightarrow{c})}^\alg{A}_{\mathbf{M}_j} = \1^\alg{A}$. Therefore $\semvalue{ \phi(\overrightarrow{d},\overrightarrow{e}, \overrightarrow{c})}^\alg{A}_{\mathbf{M}} = \1^\alg{A}$, because extensions preserve quantifier-free formulas, and then clearly $$\semvalue{ (\exists \overrightarrow{y})\phi(\overrightarrow{d},\overrightarrow{c})}^\alg{A}_{\mathbf{M}} = \1^\alg{A}.$$ We can conclude that for every $\overrightarrow{d} \in M$, $$\semvalue{ (\exists \overrightarrow{y})\phi(\overrightarrow{d}, \overrightarrow{c})}^\alg{A}_{\mathbf{M}} = \1^\alg{A},$$ and, hence, $ \semvalue{(\forall \overrightarrow{x})(\exists \overrightarrow{y}) \phi(\overrightarrow{c})}^\alg{A}_{\mathbf{M}} = \1^\alg{A}$.\qed\end{proof}

Next we provide the amalgamation result that will allow us to prove a version of Chang--\L o\'s--Suszko theorem for graded model theory.

\begin{Pro}\emph{($\exists_2$-amalgamation)}\label{am2} Let $\tuple{\textbf{A}, {\bf M}_1}$ and $\tuple{\textbf{A}, {\bf M}_2}$ be two structures for $\mathcal{P}^{\textbf{A}}$ with a common part $\tuple{\textbf{A}, {\bf M}}$ with domain generated by a sequence of elements $\overrightarrow{d}$. Moreover, suppose that  $$   \tuple{\textbf{A}, {\bf M}_2, \overrightarrow{d}} \Rrightarrow_{\exists_2} \tuple{\textbf{A}, {\bf M}_1, \overrightarrow{d}}. $$ Then, there is a structure $\tuple{\textbf{A}, {\bf N}}$ into which $\tuple{\textbf{A}, {\bf M}_2}$ can be strongly embedded by $\tuple{f, g}$ preserving all $\forall_1$-formulas, while $\tuple{\textbf{A}, {\bf M}_1}$ is $\mathcal{P}^{\textbf{A}}$-elementarily strongly embedded (taking isomorphic copies, we may assume that $\tuple{\textbf{A}, {\bf M}_1}$ is just a  $\mathcal{P}^{\textbf{A}}$-elementary substructure).  The situation is described by the following picture:
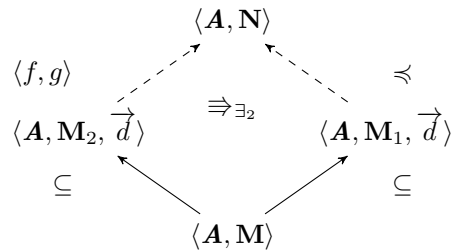
\begin{center}
\begin{tikzpicture}
    \node (E) at (0,0) {};
    \node[above=of E] (F) {$\tuple{\textbf{A}, {\bf N}}$};
        \node[below=of E] (C) {$\tuple{\textbf{A}, {\bf M}}$};

 \node[midway, below, left=of E] (A) {$\tuple{\textbf{A}, {\bf M}_2, \overrightarrow{d}}$};
  \node[midway, below, right=of E] (B) {$\tuple{\textbf{A}, {\bf M}_1, \overrightarrow{d}}$};
    \node[below=of E] (Asubt) {};
  
    \draw[->, dashed] (A)--(F) node [midway, below,  left=of E] {$\tuple{f, g}$ };
\draw[->, dashed] (B)--(F) node [midway, below, right=of E] {  $\preccurlyeq$};
\draw  (B) (A) node [midway,above] {$\Rrightarrow_{\exists_2}$};
\draw[->] (C)--(B) node [midway, below, right=of E] {  $\subseteq$};
\draw[->] (C)--(A) node [midway, below, left=of E] {  $\subseteq$};
\end{tikzpicture}
\end{center}
Moreover, the result is also true when $\tuple{\textbf{A}, {\bf M}_1}$ and $\tuple{\textbf{A}, {\bf M}_2}$ have no common part.
\end{Pro}

\begin{proof} Let $\mbox{Diag}_{\forall_1}(\emph{\textbf{A}}, {\bf M}_2)$ be the collection of all $\forall_1$-formulas in the language of the diagram of $\tuple{\emph{\textbf{A}}, {\bf M}_2}$ (where we let the elements of the domain serve as constants to name themselves) that hold in said structure. It is not a difficult to show that $  \mbox{ElDiag}(\emph{\textbf{A}}, {\bf M}_1) \cup \mbox{Diag}_{\forall_1}(\emph{\textbf{A}}, {\bf M}_2)$ (where again we let the elements of the domain serve as constants to name themselves) has a model, which suffices for the purposes of the result. For suppose  otherwise, that is, for some finite $$\mbox{Diag}_{\forall_1 0}(\emph{\textbf{A}}, {\bf M}_2) \subseteq \mbox{Diag}_{\forall_1}(\emph{\textbf{A}}, {\bf M}_2),$$ we have that $$ \mbox{ElDiag}(\emph{\textbf{A}}, {\bf M}_1) \vDash (\bigwedge  \mbox{Diag}_{\forall_1 0}(\emph{\textbf{A}}, {\bf M}_2))  \rightarrow  \overline{a}$$ for some $a \neq \1^\emph{\textbf{A}}$ (the supremum of all the values taken by $\bigwedge  \mbox{Diag}_{\forall_1 0}(\emph{\textbf{A}}, {\bf M}_2)$ in $\emph{\textbf{A}}$). Quantifying away the new individual constants,  $$ \mbox{ElDiag}(\emph{\textbf{A}}, {\bf M}_1) \vDash (\forall \overrightarrow{x})((\bigwedge  \mbox{Diag}_{\forall_1 0}^*(\emph{\textbf{A}}, {\bf M}_2))  \rightarrow  \overline{a}),$$ so $$ \mbox{ElDiag}(\emph{\textbf{A}}, {\bf M}_1) \vDash (\exists \overrightarrow{x}) ((\bigwedge  \mbox{Diag}_{\forall_1 0}^*(\emph{\textbf{A}}, {\bf M}_2)) ) \rightarrow  \overline{a}$$

Since  $ \tuple{\emph{\textbf{A}}, {\bf M}_2, \overrightarrow{d}}  \Rrightarrow_{\exists_2}  \tuple{\emph{\textbf{A}}, {\bf M}_1, \overrightarrow{d}} $, then  $$ \tuple{\emph{\textbf{A}}, {\bf M}_2, \overrightarrow{d}} \not \models (\exists \overrightarrow{x}) (\bigwedge  \mbox{Diag}_{\forall_1 0}^*(\emph{\textbf{A}}, {\bf M}_2)),$$ which is a contradiction.\qed\end{proof}

Now we are ready to prove the promised analogue of~\cite[Theorem 1.2]{ro}.
\begin{Thm}\emph{(Chang--\L o\'s--Suszko preservation theo.)}\label{Chang-Los-Suszko}  Let $T$ be a theory and $\Phi(\overrightarrow{x})$ a set of formulas in $\pl^{\alg{A}}$. 
Then, the following are equivalent:
\begin{itemize}
 \item[(i)] $\Phi(\overrightarrow{x})$ is preserved under unions of chains of models of $T$.
 \item[(ii)]  $\Phi(\overrightarrow{x})$ is 1-equivalent modulo  $T$  to a set of\/ $\forall_2$-formulas.
 \end{itemize}
\end{Thm}

\begin{proof} Once more, we only deal with the non-trivial direction of the equivalence.  Consider $(T \cup \Phi(\overrightarrow{x}))_{\forall_2}$. We want to show that $$ T \cup (T \cup \Phi(\overrightarrow{x}))_{\forall_2} \vDash \Phi(\overrightarrow{x}),$$ which will suffice to establish the theorem. The strategy is establish that any model of $ T \cup (T \cup \Phi(\overrightarrow{x}))_{\forall_2}$ has an elementary extension which is a union of $\omega$-many models of $\Phi(\overrightarrow{x})$, so by hypothesis, $\Phi(\overrightarrow{x})$ will hold there, and hence back in our original model of $ T \cup (T \cup \Phi(\overrightarrow{x}))_{\forall_2}$.

So we start with $\tuple{\alg{A}, \struct{M}_0}$ being an arbitrary model of $ T \cup (T \cup \Phi(\overrightarrow{x}))_{\forall_2}$. Now assuming that we have $\tuple{\alg{A}, \struct{M}_i}$ which is an elementary extension of $\tuple{\alg{A}, \struct{M}_0}$. We first need to  find a model $\tuple{\alg{A}, \struct{M}_i^{\prime}}$ of the theory $T \cup \Phi(\overrightarrow{x})$ such that  $\tuple{\alg{A}, \struct{M}_i} \Rrightarrow_{\exists_2} \tuple{\alg{A}, \struct{M}_i^{\prime}}$ and then quote the $\exists_2$-amalgamation theorem to obtain a model $\tuple{\alg{A}, \struct{N}_i}$ of $T \cup \Phi(\overrightarrow{x})$ into which $\tuple{\alg{A}, \struct{M}_i}$ can be strongly embedded in such a way that all  $\forall_1$-formulas are preserved by such strong embedding.

Let $U$ be all $\exists_2$-formulas that hold in $\tuple{\alg{A}, \struct{M}_i}$. We claim  then   that $T  \cup \Phi(\overrightarrow{x}) \cup U$ has a model. Otherwise by compactness, for 
\begin{equation}
\resizebox{0.5\textwidth}{!}{$\{(\exists \overrightarrow{x}_0)(\forall \overrightarrow{y}_0) \phi_0(\overrightarrow{x}_0, \overrightarrow{y}_0), \dots,  (\exists \overrightarrow{x}_n)( \forall \overrightarrow{y}_n) \phi_0(\overrightarrow{x}_n, \overrightarrow{y}_n)\} \subseteq U$}\nonumber
\end{equation}
we have that in all models of $T$, it holds that 
\begin{equation}
\resizebox{0.5\textwidth}{!}{$  (\exists \overrightarrow{x}_0 )(\forall \overrightarrow{y}_0)  \phi_0(\overrightarrow{x}_0, \overrightarrow{y}_0) \wedge \dots \wedge (\exists \overrightarrow{x}_n) (\forall \overrightarrow{y}_n) \phi_0(\overrightarrow{x}_n, \overrightarrow{y}_n) \rightarrow \overline{a}$}\nonumber
\end{equation}
 where $a$ is the immediate predecessor of $\1^\alg{A}$ and by basic manipulations, 
\begin{equation}
\resizebox{0.5\textwidth}{!}{$ (\exists \overrightarrow{x}_0, \dots, \overrightarrow{x}_n )(\forall \overrightarrow{y}_0, \dots, \overrightarrow{y}_n )(\phi_0(\overrightarrow{x}_0, \overrightarrow{y}_0)  \wedge \dots \wedge \phi_0(\overrightarrow{x}_n, \overrightarrow{y}_n)) \rightarrow \overline{a}.$}\nonumber
\end{equation}
The latter formula must be in $(T \cup \Phi(\overrightarrow{x}))_{\forall_2}$ then, which is a contradiction.

Now,$\tuple{\alg{A}, \struct{N}_i}$ is also such that for a listing $\overrightarrow{d}$ of all the elements of $\tuple{\alg{A}, \struct{M}_i}$, $\tuple{\alg{A}, \struct{N}_i, \overrightarrow{d}} \Rrightarrow_{\exists_1} \tuple{\alg{A}, \struct{M}_i, \overrightarrow{d}}$. To prove the contrapositive,  suppose that $$\tuple{\alg{A}, \struct{M}_i, \overrightarrow{d}} \models (\exists \overrightarrow{x}) \f(\overrightarrow{x}, \overrightarrow{d}) \rightarrow \overline{a}$$ where $a$ is the immediate predecessor of $\1^\alg{A}$ in the linear order of $\alg{A}$. But then $$\tuple{\alg{A}, \struct{M}_i, \overrightarrow{d}} \models (\forall \overrightarrow{x}) (\f(\overrightarrow{x}, \overrightarrow{d}) \rightarrow \overline{a}),$$ so indeed, $$\tuple{\alg{A}, \struct{N}_i, \overrightarrow{d}} \models (\forall \overrightarrow{x}) (\f(\overrightarrow{x}, \overrightarrow{d}) \rightarrow \overline{a}),$$ and, hence, $$\tuple{\alg{A}, \struct{N}_i, \overrightarrow{d}} \models (\exists \overrightarrow{x}) \f(\overrightarrow{x}, \overrightarrow{d}) \rightarrow \overline{a}.$$

Now using the existential amalgamation theorem we can obtain a structure  $\tuple{\alg{A}, \struct{M}_{i+1}}$ as an elementary extension of $\tuple{\alg{A}, \struct{M}_{i}}$ into which  $\tuple{\alg{A}, \struct{N}_{i}}$ can be strongly embedded. Now just take the union $\tuple{\alg{A}, \bigcup_{i\in \omega} \struct{M}_i}= \tuple{\alg{A}, \bigcup_{i\in \omega} \struct{N}_i}$ and apply Theorem~\ref{t:Unions}.\qed
\end{proof}

As a consequence, we can again obtain a result for single formulas, using the compactness of the consequence relation.

\begin{Cor} Let $T$ be a theory in $\pl^{\alg{A}}$ and $\f$ a  formula. 
Then, the following are equivalent:
\begin{itemize}
 \item[(i)] $\f$ is preserved under unions of chains of models of $T$.
 \item[(ii)]  $\f$ is 1-equivalent modulo $T$  to a set of\/ $\forall_2$-formulas.
 \end{itemize}
\end{Cor}

\section{Conclusions}\label{con}
In this paper we have provided some necessary steps in the systematic study of syntactic characterizations of classes of graded structures and their corresponding preservations theorems. Work in progress in the same line includes the study of the universal Horn fragment of predicate fuzzy logics and the classes axiomatized by sets of Horn clauses. Moreover, in the general endeavor of graded model theory, we believe that, among others, future works should focus on the study of types, with the construction of saturated models and type-omission theorems, the study of particular kinds of graded structures that are relevant for computer science applications and, also, the development of Lindstr\"om-style characterization theorems for  predicate fuzzy logics that may lead to the creation of a non-classical abstract model theory.

\begin{acknowledgements}
The authors are indebted to two anonymous referees and to the editor for their critical and interesting remarks that have helped improving the presentation the paper.
\end{acknowledgements}

\section*{Compliance with ethical standards}
\noindent{\bf Funding:} Costa, Dellunde, and Noguera received funding from the European Union's Horizon 2020 research and innovation program under the Marie  Curie grant agreement No 689176 (SYSMICS project). Badia  is supported by the project I 1923-N25 (\emph{New perspectives on residuated posets}) of the Austrian Science Fund (FWF).  Costa is also supported by the grant for the recruitment of early-stage research staff (FI-2017) from the Generalitat de Catalunya. Dellunde is also partially supported by the project RASO TIN2015-71799-C2-1-P, CIMBVAL TIN2017-89758-R, and the grant 2017SGR-172 from the Generalitat de Catalunya. The research leading to these results has received funding from  
AppPhil-RecerCaixa. Finally, Noguera is also supported by the project GA17-04630S of the Czech Science Foundation (GA\v{C}R).

\smallskip

\noindent{\bf Conflict of interest:} The authors declare they have no conflict of interest.

\smallskip 

\noindent This article does not contain any studies with human participants or animals performed by
any of the authors.

\end{document}